\documentclass[a4paper]{article}

\usepackage{amsmath,amssymb,amsthm}
\usepackage[titletoc,toc]{appendix}

\usepackage{vmargin}
\setmarginsrb{3cm}{1.5cm} 
                      {2cm}{2cm}              
                      {10pt}{10pt}       
                      {12pt}{30pt}       

\usepackage{hyperref}
\usepackage{algorithm}
\usepackage{algorithmic}
\usepackage[dvips]{graphicx}
\usepackage{multicol}
\usepackage{float}
\usepackage{listings}
\usepackage{subfigure}
\usepackage{enumerate}

\usepackage[centerlast, small]{caption}

\usepackage{fancyhdr} 

\def\eps{\varepsilon}

\newtheorem{theorem}{Theorem}

\newtheorem{lemma}{Lemma}
\theoremstyle{definition}

\def\Oh{\mathcal{O}}

\begin{document}

\title{Cholesky factorisation of linear systems coming from finite
  difference approximations of singularly perturbed problems}
\author{Th\'{a}i Anh Nhan\thanks{This author's work is supported by the Irish Research
Council under Grant No. RS/2011/179.}~ and Niall Madden\\
School of Mathematics,
Statistics and Applied Mathematics,\\ National University of Ireland,
Galway.\\
\small{Emails:  a.nhan1@nuigalway.ie,Niall.Madden@nuigalway.ie}}

 \maketitle
\begin{abstract}
We consider the solution of large linear systems of
  equations that arise when two-dimensional
singularly perturbed reaction-diffusion equations are discretized.
Standard methods for these problems, such as central finite
differences, lead to system matrices that are positive definite. The
direct solvers of choice for such systems are based on Cholesky
factorisation. However, as  observed in \cite{MaMa13}, these solvers
may exhibit  poor performance for singularly perturbed problems. We
provide an analysis of the distribution of entries in the factors
based on their magnitude that explains this phenomenon, and give
bounds on the ranges of the perturbation and discretization
parameters where poor performance is to be expected.
\end{abstract}
\section{Introduction}\label{sec:intro}

We consider the singularly perturbed two dimensional
reaction-diffusion problem:
\begin{equation}\label{eq:2DRD}
-\varepsilon^2 \Delta u + b(x,y)u = f(x,y), \quad \Omega=(0,1)^2,
\quad
  u(\partial\Omega)=g(x,y),
\end{equation}
where the  ``perturbation parameter'', $\varepsilon$, is a small and
positive, and the functions $g$, $b$ and $f$ are given, with
$b(x,y)\ge \beta^2>0$.

We are interested in the numerical solution of \eqref{eq:2DRD} by
the following standard finite difference technique. Denote the mesh
points of an arbitrary rectangular mesh as $(x_i, y_j)$ for $i,j \in
\{0,1, \dots, N\}$, write the local mesh widths as $h_i=x_i-x_{i-1}$
and $k_j=y_j-y_{j-1}$, and let $\bar{h}_i = (x_{i+1} - x_{i-1})/2$,
and $\bar{k}_j = (y_{j+1} - y_{j-1})/2$. Then the linear system for
the finite difference method can be written as
\begin{subequations}
\label{eq:finite diff}
\begin{equation}\label{eq:lin_sys_2DRD}
AU^N=f^N_{}, \quad \text{ where } \quad
A=\left(-\varepsilon^2\Delta^N +
\bar{h}_i\bar{k}_jb(x_i,y_j)\right),
\end{equation}
and $\Delta^N_{}$ is the  symmetrised 5-point second order central
difference operator
\begin{equation}\label{eq:operator}
\Delta^N_{}:=\left(
\begin{matrix}
&\dfrac{\bar{h}_i}{k_{j+1}}&\\
\dfrac{\bar{k}_j}{h_{i}} &
-\left(\bar{k}_j\left(\dfrac{1}{h_i}+\dfrac{1}{h_{i+1}}\right)
+ \bar{h}_i\left(\dfrac{1}{k_j}+\dfrac{1}{k_{j+1}}\right)\right) &\dfrac{\bar{k}_j}{h_{i+1}}\\
&\dfrac{\bar{h}_i}{k_{j}}&
\end{matrix}
\right).
\end{equation}
\end{subequations}
It is known that the scheme (\ref{eq:finite diff})  applied to
(\ref{eq:2DRD}) on a boundary layer-adapted mesh with $N$ intervals
in each direction yields a parameter robust approximation, see,
e.g., \cite{ClGr05a,Linss10}.
Since $A$ in  \eqref{eq:lin_sys_2DRD} is a banded, symmetric and
positive definite, the direct solvers of choice are variants on
Cholesky factorisation. This is based on the idea that there exists
a unique lower-triangular matrix $L$ (the ``Cholesky factor'') such
that $A=LL^T$ (see, e.g.,~\cite[Thm. 4.25]{Golub_96}). Conventional
wisdom is that the computational complexity of these methods depends
exclusively on $N$ and the structure of the matrix (i.e., its
sparsity pattern). However, MacLachlan and Madden  \cite[\S
4.1]{MaMa13} observe that standard implementations of Cholesky
factorisation
 applied to  (\ref{eq:lin_sys_2DRD}) perform poorly when  $\eps$
in \eqref{eq:2DRD} is small. Their explanation is that the Cholesky
factor, $L$, contains many small entries that fall in to the range
of \emph{subnormal  floating-point
  numbers}. These are  numbers that have
 magnitude  (in exact arithmetic)
between $2^{-1074} \approx 5 \times 10^{-324}$  and $2^{-1022}
\approx 2\times 10^{-308}$
 (called  \texttt{realmin} in MATLAB). Numbers greater than
$2.2\times 10^{-308}$  are represented  faithfully in IEEE
 standard double precision, while numbers less than  $2^{-1074}$ are
 flushed to zero (we'll call such numbers
 ``\emph{underflow-zeros}''). Floats between these values do not have
 full
 precision, but allow for  ``gradual underflow'', which (ostensibly)
 leads to more reliable
 computing (see, e.g., \cite[Chap. 7]{Overton01}). Unlike standard
 floating-point numbers, most CPUs do not support operations on
 subnormals directly, but rely on microcode implementations, which are
 far less efficient. Thus it is to be expected that it is more
 time-consuming to factorise $A$ in \eqref{eq:lin_sys_2DRD} when
 $\eps$ is small.
As an example of this,  consider \eqref{eq:finite diff} where
$N=128$ and the mesh is uniform. The nonzero entries of the
associated Cholesky factor are located on the diagonals that are at
most a distance $N$ from main diagonal. In \autoref{fig:diag_L}, we
plot the absolute value of largest entry of  a given diagonal of
$L$, as a function of its distance from the main diagonal. On the
left of \autoref{fig:diag_L}, where  $\eps=1$, we observe that the
magnitude of the largest entry gradually decays away from the
location of the nonzero entries of $A$.  In contrast, when
$\eps=10^{-6}$ (on the right), magnitude of the largest entry decays
exponentially.

To demonstrate the effect of this on computational efficiency, in
\autoref{tab:CholUniform512} we show the time, in seconds,  taken to
compute the factorisation of  $A$ in (\ref{eq:lin_sys_2DRD}) with a
uniform mesh, and $N=512$, on a single core of  AMD Opteron 2427,
2200 MHz processor, using
 CHOLMOD~\cite{YChen_etal_2008a} with ``natural order'' (i.e., without
 a  fill reducing ordering).
Observe that the time-to-factorisation increases from 52 seconds
when $\eps$ is large, to nearly 500 seconds when $\eps=10^{-3}$,
when over 1\% of the entries are in the subnormal range. When $\eps$
is smaller again, the number of nonzero entries in $L$ is further
reduced, and so the execution time decreases as well.

\begin{figure}[htb]
\begin{center}
\includegraphics[width=4.75cm]{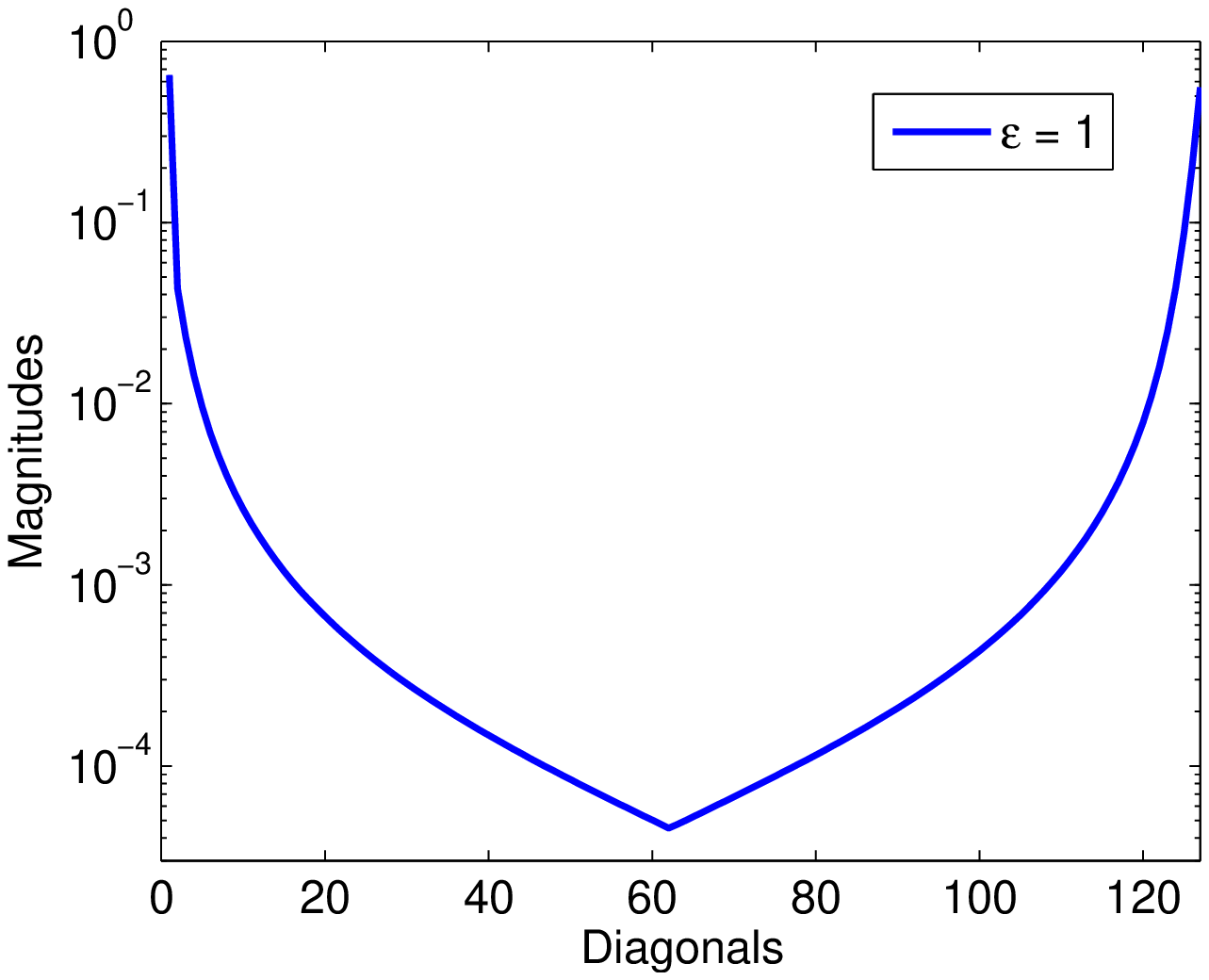}
\hspace{.5cm}
\includegraphics[width=4.75cm]{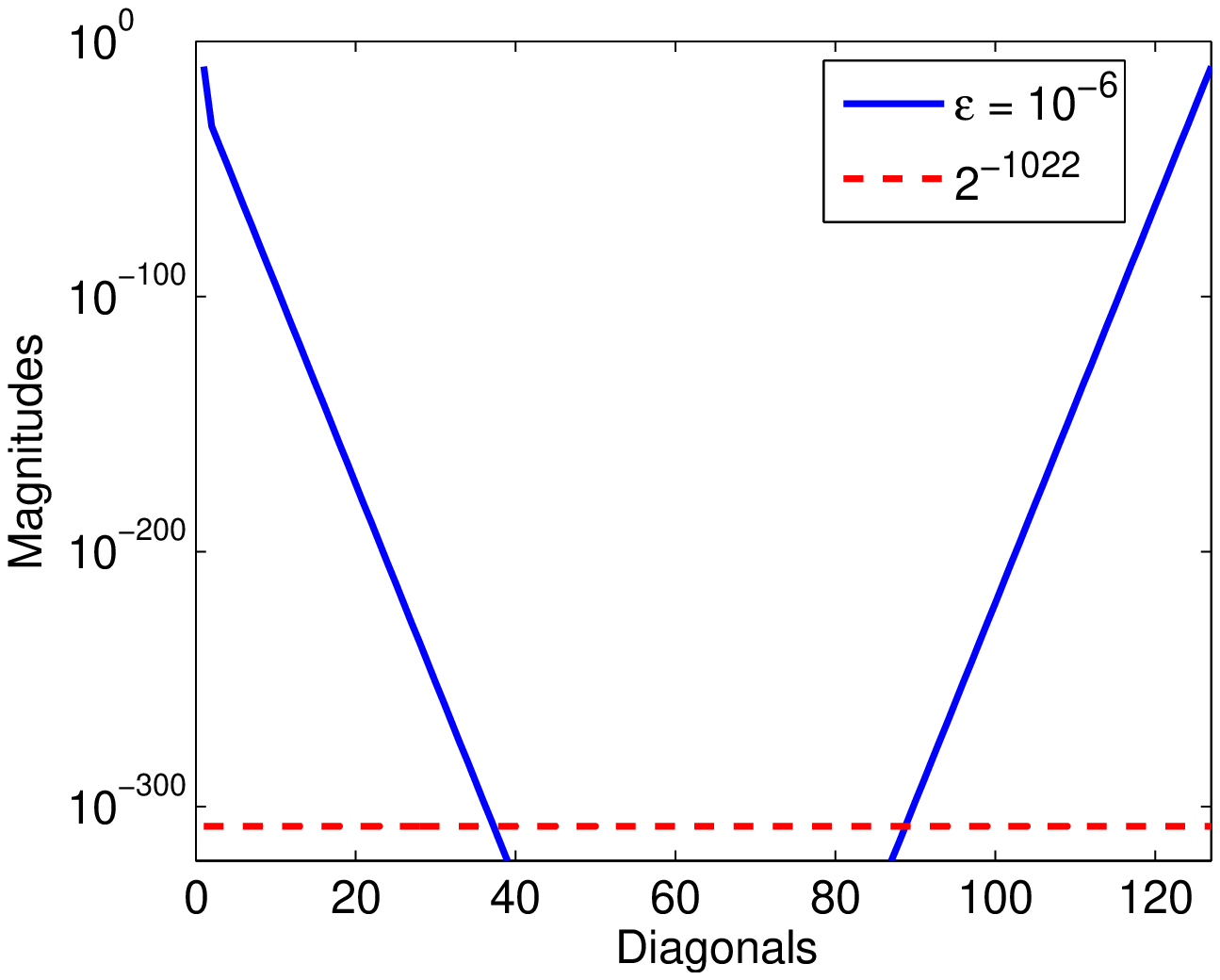}
\caption{Semi-log plot of maximal  entries on diagonals of $L$ with
$N=128$, and $\eps=1$ (left) and $\eps=10^{-6}$ (right).}
\label{fig:diag_L}
\end{center}
\end{figure}


\begin{table}[htb]
\centering
\begin{tabular}{c||r|r|r|r|r|r}
\multicolumn{1}{c|}{$\eps$} &
  \multicolumn{1}{|c|}{$10^{-1}$} &
  \multicolumn{1}{|c|}{$10^{-2}$} &
  \multicolumn{1}{|c|}{$10^{-3}$} &
  \multicolumn{1}{|c|}{$10^{-4}$} &
  \multicolumn{1}{|c|}{$10^{-5}$} &
  \multicolumn{1}{|c}{$10^{-6}$} \\\hline
Time (s)  &         52.587 &       52.633 &      496.887 &      175.783 &       74.547 &       45.773 \\
Nonzeros in $L$  & 133,433,341 &133,433,341& 128,986,606 & 56,259,631&33,346,351 & 23,632,381\\
Subnormals in $L$ &                  0 &          0&    1,873,840 &
2,399,040 &    1,360,170 &     948,600\\
Underflow zeros &                  0 &           0&  4,446,735 &
77,173,710 & 100,086,990 &
 109,800,960\\
\end{tabular}
\caption{Time taken (in seconds) to compute the Cholesky factor,
$L$,
  of $A$ in (\ref{eq:finite diff}) on a uniform mesh with
  $N=512$. The number of nonzeros, subnormals, and underflow-zeros in $L$ are~also~shown.}\label{tab:CholUniform512}
\end{table}


Our goal is to give an analysis that fully explains the observations
of \autoref{fig:diag_L} and \autoref{tab:CholUniform512}, and that
can also be exploited in other solver strategies. We derive
expressions, in terms of $N$ and $\eps$, for the magnitude of
entries of  $L$ as determined by their location. Ultimately, we are
interested in the analysis of systems that arise from the numerical
solution of  \eqref{eq:2DRD} on appropriate boundary layer-adapted
meshes. Away from the boundary, such meshes are usually uniform.
Therefore, we begin in Section~\ref{sec:mag_entries} with studying a
uniform mesh discretisation, in the setting of exact arithmetic,
which provides mathematical justification  for observations in
\autoref{fig:diag_L}. In Section~\ref{sec:distribution_fill_in},
this analysis is used to quantify to number of  entries in the
Cholesky factors of a given magnitude. As an application of this, we
show how to determine the number of subnormal numbers that will
occur in $L$ in a floating-point setting, and also determine an
lower  bound for $\eps$ for which the factors are free of subnormal
numbers. Finally, the Cholesky factorisation on a boundary
layer-adapted mesh is discussed in Section
\ref{sec:Cholesky_Shishkin}, and our conclusions are summarised in
Section \ref{sec:conclusion}.

\section{Cholesky factorisation on a uniform mesh}\label{sec:Cholesky}
\subsection{The magnitude of the fill-in entries}\label{sec:mag_entries}
We consider the discretisation~\eqref{eq:operator} of the model
problem \eqref{eq:2DRD} on a uniform mesh with $N$ intervals on each
direction. The equally spaced stepsize is denoted by $h=N^{-1}$.
When $\eps\ll h$, which is typical in a singularly perturbed regime,
the system matrix in \eqref{eq:lin_sys_2DRD} can be written as the
following 5-point stencil
\begin{equation}\label{eq:matrix_uniform}
A=\begin{pmatrix} &-\eps^2&\\
-\eps^2\ &4\eps^2+h^2b(x_i,y_j)\ &-\eps^2\\
&-\eps^2&
\end{pmatrix}=\begin{pmatrix} &-\eps^2&\\
-\eps^2\ &\Oh{(h^2)}\ &-\eps^2\\
&-\eps^2&
\end{pmatrix},
\end{equation}
since $(4\eps^2+h^2b(x_i,y_j))=\Oh(h^2)$, where we write
$f(\cdot)=\Oh{(g(\cdot))}$ if there exist positive constants $C_0$
and $C_1$, independent of $N$ and $\eps$,  such that
$C_0|g(\cdot)|\le~f(\cdot)\le~C_1|g(\cdot)|$.

Algorithm~\ref{al:Gaxpy_Cholesky_2} presents a version of Cholesky
factorisation which  adapted from~\cite[page 143]{Golub_96}. It
computes a lower triangular matrix $L$ such that $A=LL^T$. We will
follow MATLAB notation by denoting $A=[a(i,j)]$ and $L=[l(i,j)]$.
\begin{algorithm}
\caption{Cholesky factorisation:}\label{al:Gaxpy_Cholesky_2}
\textsf{
\textbf{for} $j = 1:n$\\
\phantom{class}\textbf{if} $j=1$\\
\phantom{classclassc}\textbf{for} $i=j:n$\\
\phantom{classclassclass}$l(i, j) = \dfrac{a(i,j)}{\sqrt{a(j,j)}}$\\
\phantom{classclass}\textbf{end}\\
\phantom{class}\textbf{elseif} $(j>1)$\\
\phantom{classclassc}\textbf{for} $i=j:n$\\
\phantom{classclassclass}$l(i,j) = \dfrac{a(i,j)-\sum_{k=1}^{j-1} l(i,k)l(j,k)}{\sqrt{a(j,j)}}$\\
\phantom{classclass}\textbf{end}\\
\phantom{class}\textbf{end}\\
\textbf{end}}
\end{algorithm}

We set $m=N-1$, so  $A$ is a sparse, banded $m^2 \times m^2$ matrix,
with a bandwidth of $m$, and has  no more than five nonzero entries
per row. Its factor, $L$, is far less sparse: although it has the
same bandwidth as $A$,  it has  $\Oh(m)$ nonzeros per row (see,
e.g., \cite[Prop. 2.4]{Demmel97}). The set of non-zero entries in
$L$ that are zero in the corresponding location in $A$ is called the
\emph{fill-in}. We want to find a recursive way to express the
magnitude of these fill-in entries, in terms of $\eps$ and $h$.

To analyse the magnitude of the fill-in entries, we borrow notation
from  \cite[Sec. 10.3.3]{Saad_03}, and
 form
distinct sets denoted $L^{[0]}$, $L^{[1]}, \dots, L^{[m]}$ where all
entries of $L$ of the same magnitude (in a sense explained carefully
below) belong to the same set. We denote by $l^{[k]}$ the magnitude
of entries in $L^{[k]}$, i.e., $l(i,j)\in L^{[k]}$ if and only if
$l(i,j)$ is $\Oh(l^{[k]})$. We shall see that these sets are quite
distinct, meaning that $l^{[k]}\gg l^{[k+1]}$ for $k\ge 1$.
$L^{[0]}$ is used to denote the set of nonzero entries in $A$, and
entries of $L$ that are zero (in exact arithmetic) are defined to
belong to $L^{[\infty]}$.

In Algorithm \ref{al:Gaxpy_Cholesky_2}, all the entries of $L$ are
initialised as zero, and so belong to $L^{[\infty]}$. Suppose that
$p_{i,j}$ is such that $l(i,j)\in L^{[p_{i,j}]}$, so, initially,
each $p_{i,j}=\infty$. At each sweep through the algorithm, a new
value of $l(i,j)$ is computed, and so $p_{i,j}$ is modified. From
line 8 in Algorithm~\ref{al:Gaxpy_Cholesky_2}, we can see that the
$p_{i,j}$ is updated by
\begin{equation*}\label{eq:def_level}
p_{i,j}^{}=\begin{cases}\min\{0, p_{i,1}^{}+p_{j,1}^{}+1,
p_{i,2}^{}+p_{j,2}^{}+1,
\ldots, p_{i,j-1}^{}+p_{j,j-1}^{}+1 \},\ \textrm{if }\  a(i,j)\neq0,\\
\min\{p_{i,1}^{}+p_{j,1}^{}+1, p_{i,2}^{}+p_{j,2}^{}+1, \ldots,
p_{i,j-1}^{}+p_{j,j-1}^{}+1 \},\ \textrm{otherwise}.
\end{cases}
\end{equation*}
Then, as we shall explain in detail below, it can be determined that
$L$ has a block structure shown in
\eqref{fig:L_structure}--\eqref{fig:form_L_partition_2}, where, for
brevity,  the entries belonging to $L^{[k]}$ are denoted by $[k]$,
and the entries that corresponding to nonzero entries of original
matrix are written in terms of their magnitude:
\begin{small}
\begin{subequations}
\label{eq:partition L}
\begin{equation}\label{fig:L_structure}
L=\begin{pmatrix} M&&&&\\
P&Q&&&\\
&P&Q&&\\
&&\ddots&\ddots&\\
&&&P&Q\\
\end{pmatrix},
\text{ where }
M=\begin{pmatrix}
\Oh{(h)}&&&&\\
\Oh{(\eps^2/h)}&\Oh{(h)}&&&\\
&\Oh{(\eps^2/h)}&\Oh{(h)}&&\\
&&\ddots&\ddots&\\
&&&\Oh{(\eps^2/h)}&\Oh{(h)}\\
\end{pmatrix},
\end{equation}
\begin{equation}\label{fig:form_L_partition_1}
P=\begin{pmatrix}
\Oh{(\eps^2/h)}&[1]&[2]&[3]&\ldots&[m-2]&[m-1]\\
&\Oh{(\eps^2/h)}&[1]&[2]&\ldots&[m-3]&[m-2]\\
&&\ddots&\ddots&\ddots&\vdots&\vdots\\
&&&\Oh{(\eps^2/h)}&[1]&[2]&[3]\\
&&&&\Oh{(\eps^2/h)}&[1]&[2]\\
&&&&&\Oh{(\eps^2/h)}&[1]\\
&&&&&&\Oh{(\eps^2/h)}\\
\end{pmatrix},
\end{equation}
\begin{equation}\label{fig:form_L_partition_2}
Q=\begin{pmatrix}\Oh{(h)}&&&&&&\\
\Oh{(\eps^2/h)}&\Oh{(h)}&&&&&\\
[3]&\Oh{(\eps^2/h)}&\Oh{(h)}&&&&\\
[4]&[3]&\Oh{(\eps^2/h)}&\Oh{(h)}&&&\\
\vdots&\vdots&\ddots&\ddots&\ddots&&\\
[m-1]&[m-2]&\ldots&[3]&\Oh{(\eps^2/h)}&\Oh{(h)}&\\
[m]&[m-1]&\ldots&[4]&[3]&\Oh{(\eps^2/h)}&\Oh{(h)}\\
\end{pmatrix}.
\end{equation}
\end{subequations}
\end{small}%
We now explain why the entries of $L$, which are computed by column,
have the structure shown in \eqref{eq:partition L}. According to
Algorithm \ref{al:Gaxpy_Cholesky_2}, the first column of $L$ is
computed by $l(i,1)=a(i,1)/\sqrt{a(1,1)}$, which shows that there is
no fill-in entry in this column. For the second column, the only
fill-in entry is
\[
l(m+1,2)=\frac{a(m+1,2)-l(m+1,1)l(2,1)}{\sqrt{a(2,2)}}=
\frac{0-\Oh{(\eps^2/h)}\Oh{(\eps^2/h)}}{\Oh{(h)}}=\Oh{(\eps^4/h^3)},
\]
where $l(m+1,1)$ and $l(2,1)$ belong to $L^{[0]}$, so $l(m+1,2)$ is
in $L^{[1]}$. Similarly, there are two fill-ins in third column:
$l(m+1,3)$ and $l(m+2,3)$.  The entry $l(m+1,3)$ is computed as
\[
\begin{split}
l(m+1,3)&=\frac{a(m+1,3)-\sum_{k=1}^{2}l(m+1,k)l(3,k)}{\sqrt{a(3,3)}}
=\frac{-l(m+1,2)l(3,2)}{\sqrt{a(3,3)}}
\end{split}
\]
which is $\Oh{(\eps^6/h^5)}$; moreover, since $l(m+1,2) \in
L^{[1]}$, and $l(3,2) \in L^{[0]}$, so $l(m+1,3) \in L^{[2]}$.
Similarly, it is easy to see that $l(m+2,3) \in L^{[1]}$. We may now
proceed by induction to show that
$l(m+1,j+1)=\Oh{(\eps^{2(j+1)}/h^{(2j+1)})}$ belongs to $L^{[j]}$,
for $1\le j\le m-2$. Suppose
$l(m+1,j)=\Oh{(\eps^{(2j)}/h^{(2j-1)})}\in L^{[j-1]}$. Then
\[
\begin{split}
l(m+1,j+1)&=\frac{a(m+1,j+1)-\sum_{k=1}^{j}l(m+1,k)l(j+1,k)}{\sqrt{a(j,j)}}\\
&=\frac{-l(m+1,j)l(j+1,j)}{\sqrt{a(j,j)}}, \quad\textrm{since } l(j+1,k)=0,\quad \forall k\le j-1,\\
&=\frac{\Oh{(\eps^{(2j)}/h^{(2j-1)})}\Oh{(\eps^2/h)}}{\Oh{(h)}}=\Oh{(\eps^{(2j+2)}/h^{(2j+1)})}.\\
\end{split}
\]
And, because $l(j+1,j)\in L^{[0]}$, we can deduce that
$l(m+1,j+1)\in L^{[j]}$. The process is repeated from column 1 to
column $m$, yielding the  pattern for $P$ shown in
\eqref{fig:form_L_partition_1}.

A similar process is used to show that $Q$ is as given in
\eqref{fig:form_L_partition_2}. Its first fill-in entry is
$l(m+3,m+1)$.
Note that $a(m+3,m+1)=l(m+1,1)=l(m+1,2)=0$, that the magnitude of
the entry in $L^{[j]}$ is $\Oh{(\eps^{2(j+1)}/h^{(2j+1)})}$, and
that the sum of two entries of the different magnitude  has the same
magnitude as larger one. Then
\[
\begin{split}
l(m+3,m+1)&=\frac{-\sum_{k=3}^{m}l(m+3,k)l(m+1,k)}{\sqrt{a(m+1,m+1)}}\\
&=\left[\Oh{\left(\frac{\eps^2}{h}\right)}\Oh{\left(\frac{\eps^6}{h^5}\right)}+\Oh{\left(\frac{\eps^4}{h^3}\right)}\Oh{\left(\frac{\eps^8}{h^7}\right)}+\ldots\right.\\
&\qquad\left.\,\,+\Oh{\left(\frac{\eps^{2(m-2)}}{h^{(2(m-3)+1)}}\right)}
\Oh{\left(\frac{\eps^{2(m)}}{h^{(2(m-1)+1)}}\right)}\right]\frac{1}{\Oh{(h)}}\\
&=\left[\Oh{\left(\frac{\eps^2}{h}\right)}\Oh{\left(\frac{\eps^6}{h^5}\right)}\right]\frac{1}{\Oh{(h)}}=\Oh\left(\frac{\eps^8}{h^7}\right),
\end{split}
\]
and so $l(m+3,m+1)$ belongs to $L^{[3]}$. Proceeding inductively, as
was done for $P$ shows that $Q$ has the form given in
(\ref{fig:form_L_partition_2}). Furthermore, the same process
applies to each block of $L$ in  \eqref{fig:L_structure}.
 Summarizing, we have established the following result.
\begin{theorem}\label{thm:mag_entries}
The fill-in entries of the Cholesky factor $L$ of the matrix $A$
defined in~(\ref{eq:matrix_uniform}) is as given in
(\ref{eq:partition L}). Moreover, setting $\delta=\eps/h$, the
magnitude $l^{[k]}$ is
\begin{equation}\label{eq:entries_Cholesky_l_2}
l^{[k]}=\Oh{\left(\eps^{2(k+1)}/h^{(2k+1)}\right)}=\Oh{\left(\delta^{2(k+1)}_{}h\right)}
\quad \text{ for } \quad k=1,2, \dots, m.
\end{equation}
\end{theorem}

\subsection{Distribution of fill-in entries in a floating-point
setting}\label{sec:distribution_fill_in} In practice, Cholesky
factorisation is computed in a floating-point setting.
As discussed in Section \ref{sec:intro}, the time taken to compute
these factorisations increases greatly if there are many subnormal
numbers present. Moreover, even the underflow-zeros in the factors
can  be expensive to compute, since they typically arise from
intermediate calculations involving subnormal numbers. Therefore, in
this section we use the analysis of Section \ref{sec:mag_entries},
to estimate, in terms of $\eps$ and $N$,  the number of entries in
$L$ that are of a given magnitude. From this, one can easily predict
the number of subnormals and underflow-zeros in $L$.
\begin{lemma}\label{lem:NNZ_Cholesky}
Let $A$ be the  $m^2\times m^2$  matrix in \eqref{eq:finite diff}
where the mesh is uniform. Then the number of nonzero entries in the
Cholesky factor $L$ (i.e., $A=LL^T$) \emph{computed using exact
arithmetic} is
\begin{equation}\label{eq:NNZ_Cholesky}
L_{nz}^{}=m^3+m-1.
\end{equation}
\end{lemma}
\begin{proof}
Since $A$ has  bandwidth $m$, and so too does $L$ (\cite[Prop.
2.3]{Demmel97}). By the Algorithm~\ref{al:Gaxpy_Cholesky_2}, the
fill-in entries only occur from row $(m+1)$. So, from row $(m+1)$,
any row of $L$ has $(m+1)$ nonzero entries and there are $m(m-1)$
such rows, plus $2m-1$ nonzero entries from  top-left block $M$ in
\eqref{fig:L_structure}. Summing these values, we
obtain~\eqref{eq:NNZ_Cholesky}.
\end{proof}
Let $|L^{[k]}|$ be the number of fill-in entries which belong to
$L^{[k]}$. To estimate $|L^{[k]}|$, it is sufficient to evaluate the
fill-in entries in the submatrices $P$ and $Q$ shown in
\eqref{eq:partition L}. \autoref{tab:fill_ins_and_number} describes
the number of fill-in entries associated with their magnitude.
\begin{table}[H]
\centering
\begin{tabular}{c||c|c|||c}
$L^{[k]}$&   $|L^{[k]}|$ in $P $     & $|L^{[k]}|$ in $ Q$   & $|L^{[k]}|$ in $ [P,Q]$ \\
\hline
$L^{[1]}$& $m-1$  & 0 & $m-1$\\
$L^{[2]}$& $m-2$  & 0 &$m-2$\\
$L^{[3]}$& $m-3$  & $m-2$ &$2m-5$\\
\vdots& \vdots &  \vdots   & \vdots\\
$L^{[k]}$& $m-k$  & $m-k+1$ &$2m-2k+1$ \\
\vdots& \vdots &  \vdots  &  \vdots\\
$L^{[m-2]}$& 2  & $3$ &5\\
$L^{[m-1]}$& 1  & $2$ &3\\
$L^{[m]}$& 0 &  1  &1\\
\end{tabular}
\caption{Number of fill-in entries in $P$ and $Q$ associated with
their magnitude.}\label{tab:fill_ins_and_number}
\end{table}
Note that there are $(m-1)$ blocks like $[P,Q]$ in $L$. Then, since
$l^{[k]} \ll l^{[k-1]}$, and the smallest (exact) nonzero entries
belong to $L^{[m]}$ we can use \autoref{tab:fill_ins_and_number} to
determine the number of entries that are at most $\Oh(l^{[p]})$, for
some given $p$ as:
\[ 
\sum\limits_{k=p}^{m}|L^{[k]}|=
\begin{cases} (m-1)(2m-3)+(m-1)(m-2)^2=(m-1)^3 & p=1,\\
(m-1)(m-2)+(m-1)(m-2)^2=(m-2)(m-1)^2 & p=2,\\
(m-1)(m-p+1)^2 & p \ge 3.
\end{cases}
\]
These equations  can be combined and summarised as follows.
\begin{theorem}\label{thm:est_num_fill_ins}
Let $A$ be the matrix of the form \eqref{eq:matrix_uniform}. Then,
the number of fill-in entries associated with their magnitude of the
matrix $L$ satisfies
\begin{equation}\label{eq:entries_estimate}
\sum\limits_{k=p}^{m}|L^{[k]}|\le (m-1)(m-p+1)^2, \quad  p\ge 1.
\end{equation}
\end{theorem}
Combining Theorems \ref{thm:mag_entries} and
\ref{thm:est_num_fill_ins} enables us to accurately predict the
total number, and location,  of subnormal and underflow-zero entries
in  $L$, for given $N$ and $\eps$. For example, recall
\autoref{fig:diag_L} where we took $\eps=10^{-6}$ and $N=128$. To
determine, using \autoref{thm:mag_entries}, the diagonals where
entries are subnormal, we solve
\begin{equation}\label{eq:sub_eq}
(\eps N)^{2(k+1)} =  2^{-1022}N,
\end{equation}
for $N=128$ and $\eps=10^{-6}$, which yields $k\approx 38$. It
clearly agrees with the observation in \autoref{fig:diag_L}; i.e.,
the maximal value of the entries on diagonals 38 and $N-38=90$ are
less than \texttt{realmin}.
Similarly, all entries on diagonals between 40 and 88 are flushed to
zero.

As a further example, letting $N=512$ and $\eps=10^{-6}$,
by~\eqref{eq:entries_estimate}, the total number of underflow-zero
and subnormal entries in $L$ are, respectively,
\[
\quad \sum\limits_{k=48}^{511}|L^{[k]}|=109,800,960, ~  \text{ and }
~
\sum\limits_{k=46}^{47}|L^{[k]}|=\sum\limits_{k=46}^{511}|L^{[k]}|-\sum\limits_{k=48}^{511}|L^{[k]}|
=948,600.
\]
This is exactly what is observed in~Table~\ref{tab:CholUniform512}.
Moreover, the total number of entries with magnitude less than
\texttt{realmin} is 110,749,560 which is over 80\% of the exact
nonzero entries (cf. Lemma~\ref{lem:NNZ_Cholesky}) in $L$:
133,433,341. Such a predictable appearance of subnormals and
underflows is important in the sense of choosing suitable linear
solvers, i.e., direct or iterative ones.

More generally, we  can  use \eqref{eq:sub_eq} to investigate ranges
of $N$ and $\eps$ for which subnormal entries occur (assuming $\eps
\leq N^{-1}$). Since  the largest possible value of $k$ is $m$, a
Cholesky factor will have subnormal entries if $\eps$ and $N$ are
such that $(\eps N)^{2N}\leq 2^{-1022}N$. Rearranging, this gives
that
\begin{equation}\label{eq:g(N)}
\eps\leq
\frac{1}{N}\left(2^{-1022}N\right)^{1/(2N)}=2^{-511/N}N^{(1/(2N)-1)}=:g(N).
\end{equation}
The function $g$ defined in \eqref{eq:g(N)} is informative because
it gives the largest  value of  $\eps$ for a discretisation with
given  $N$ leads to a Cholesky factor with entries less than
$2^{-1022}$. For example, \autoref{fig:g(N)} (on the left) shows
$g(N)$ for $N\in [200,500]$. It demonstrates that, for  $\eps\le
1.05\times10^{-3}$ (determined numerically), subnormal entries are
to be expected for some values of
 $N$ (cf. \autoref{tab:CholUniform512}).
The line $\eps=10^{-3}$ intersects $g$ at approximately $N=263$ and
$N=484$, meaning that
 a discretisation with $263 \leq N \leq 484$  yields
 entries with the magnitude less than $2^{-1022}$ in $L$
 for $\eps=10^{-3}$.  On the right of \autoref{fig:g(N)} we
 show that, for large $N$,  $g(N)$ decays like $N^{-1}$. Since we are
 interested in the regime where
  $\eps \leq N^{-1}$, this shows that, for small $\eps$, subnormals are to
 be expected for all but the smallest values of $N$.

\begin{figure}[htb]
\begin{center}
\includegraphics[width=4.75cm]{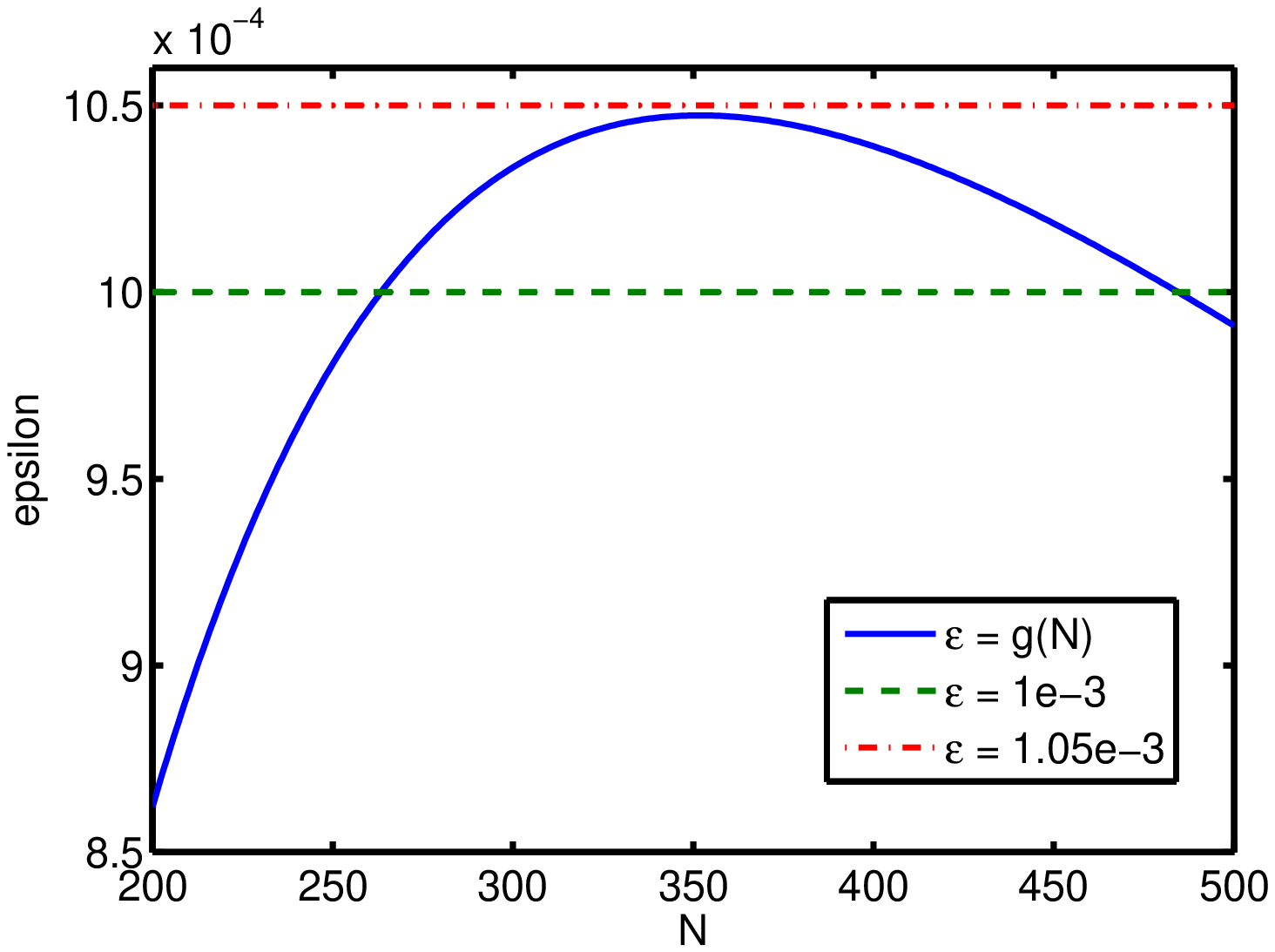}
\hspace{.5cm}
\includegraphics[width=4.75cm]{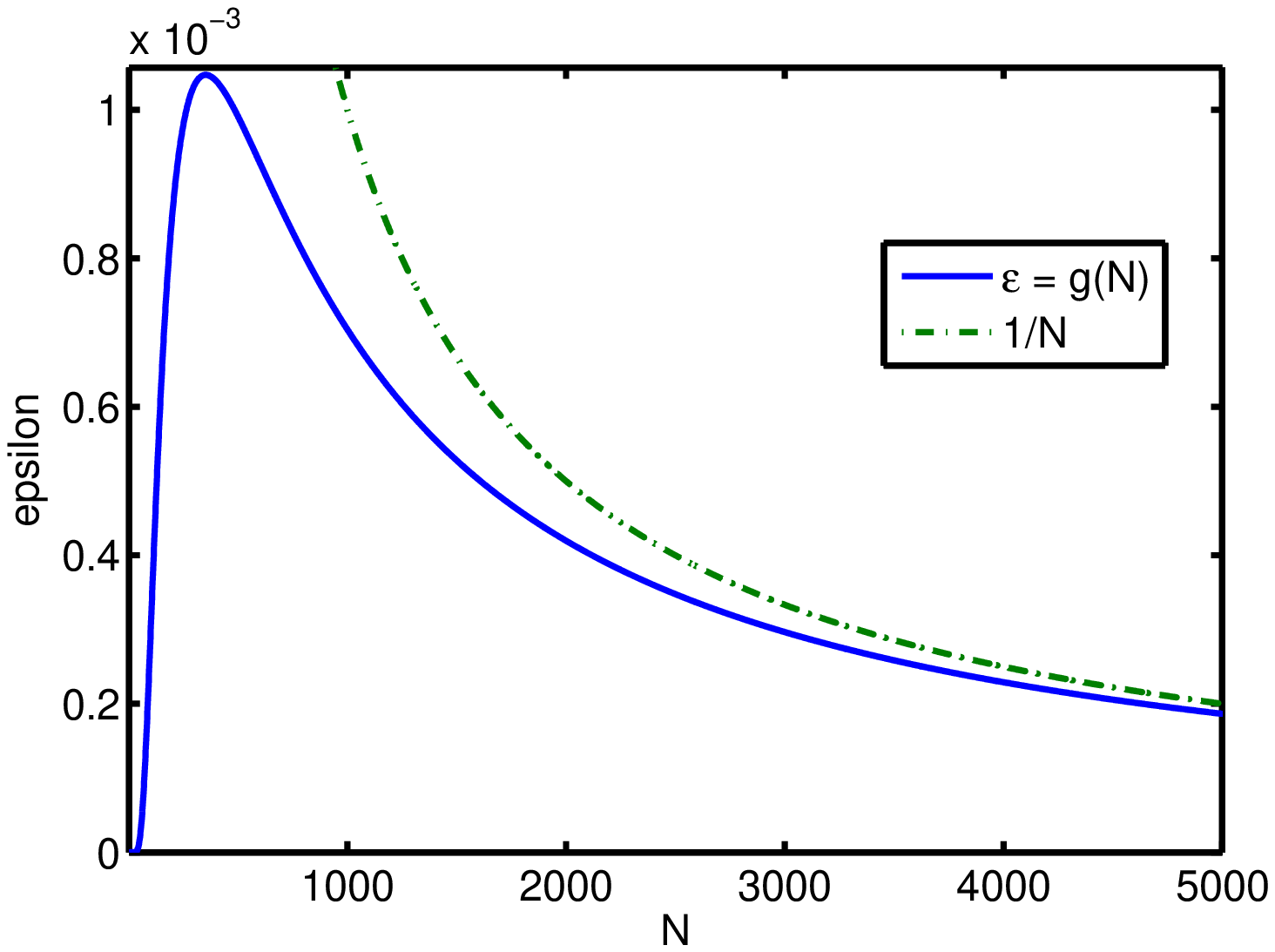}
\caption{The  function $g(N)$ defined in \eqref{eq:g(N)} with $N\in
[200,500]$ (left) and $N\in[1,5000]$ (right).} \label{fig:g(N)}
\end{center}
\end{figure}


\subsection{Boundary layer-adapted meshes}\label{sec:Cholesky_Shishkin}
Our analysis so far has been for computations on uniform meshes.
However, a scheme such as (\ref{eq:finite diff}) for (\ref{eq:2DRD})
is usually applied on a layer-adapted mesh, such as a Shishkin mesh.
For these  meshes, in the neighbourhood of the boundaries, and
especially near corner layers, the local mesh width is $\Oh(\eps
N^{-1})$ in each direction, and so the entries of the system matrix
are of the same order, and no issue with subnormal numbers is likely
to arise. However, away from layers, these fitted meshes are usually
 uniform, with a local mesh width of $\Oh(N^{-1})$, and so the
 analysis outlined above applies directly.
Since roughly one quarter (depending on mesh construction) of all
mesh points are located in this region, the influence on the
computation is likely to be substantial.

The main complication in extending our analysis to, say, a Shishkin
mesh, is in  considering the ``edge layers'', where the mesh width
may be
 $\Oh(\eps N^{-1})$ in one coordinate direction, and  $\Oh(N^{-1})$ in another. Although we have
 not analysed this situation carefully, in practise it seems that the
 factorisation behaves more like a uniform mesh.
This is demonstrated in \autoref{tab:CholShish512} below. Comparing
with \autoref{tab:CholUniform512}, we see, for small $\eps$, the
number of entries flushed to zero is roughly three-quarters that of
the uniform mesh case.
\begin{table}[htb]
\centering
\begin{tabular}{c||r|r|r|r|r|r}
\multicolumn{1}{c|}{$\eps$} &
  \multicolumn{1}{|c|}{$10^{-1}$} &
  \multicolumn{1}{|c|}{$10^{-2}$} &
  \multicolumn{1}{|c|}{$10^{-3}$} &
  \multicolumn{1}{|c|}{$10^{-4}$} &
  \multicolumn{1}{|c|}{$10^{-5}$} &
  \multicolumn{1}{|c}{$10^{-6}$} \\\hline
Time (s) &      52.580 &       58.213 &      447.533 &      179.540&      101.507 &       73.250\\
Nonzeros in $L$   &  133,433,341 &  133,240,632 &  127,533,193 &   78,091,189 &   62,082,599 &   54,497,790\\
Subnormals in $L$ &          0 &      28,282 &    2,648,308 &    1,669,345 &    1,079,992 &     814,291 \\
Underflow zeros &         0 &     192,709 &    5,900,148 &   55,342,152 &   71,350,742 &   78,935,551\\
\end{tabular}
\caption{Time taken (in seconds) to compute the Cholesky factor,
$L$,
  of $A$ in (\ref{eq:finite diff}) on a Shishkin mesh with
  $N=512$. The number of nonzeros, subnormals, and underflow-zeros in $L$ are~also~shown.}\label{tab:CholShish512}
\end{table}

\section{Conclusions}\label{sec:conclusion}
The paper addresses, in a comprehensive way, issues raised in
\cite{MaMa13} by showing how to predict the number and location of
subnormal and underflow entries in the Cholesky factors of $A$
in~\eqref{eq:lin_sys_2DRD} for given $\eps$ and $N$.

Further developments on this work are possible. In particular, the
analysis shows that, away from the existing diagonals, the magnitude
of fill-in entries decay exponentially, as seen
in~\eqref{eq:entries_Cholesky_l_2}, a fact that could be  exploited
in the  design of preconditioners of iterative solvers. For example,
as shown in \autoref{lem:NNZ_Cholesky}, the
 Cholesky factor of $A$, in exact arithmetic, has
$\mathcal{O}(N^3)$ nonzero entries. However,
Theorem~\ref{thm:est_num_fill_ins} shows that, in practice (i.e., in
a float-point setting), there are only $\mathcal{O}(N^2)$ entries in
$L$ when $\eps$ is small and $N$ is large. This suggests that, for a
singularly perturbed problem, an incomplete Cholesky factorisation
may be a very good approximation for $L$. This is a topic of ongoing
work.

In this  paper we have restricted our  study to Cholesky
factorisation of the coefficient matrix arising from
finite-difference discretisation of the model problem
(\ref{eq:2DRD}) on a uniform and a boundary layer-adapted mesh, the
same phenomenon is also observed in more general context of
singularly perturbed problems. That includes the $LU$-factorisations
of the coefficient matrices coming from both finite difference and
finite element methods applied to reaction-diffusion and
convection-diffusion problems, though further investigation is
required to establish the details.


\end{document}